\documentclass{elsarticle}

\usepackage[centertags]{amsmath}
\usepackage{amsfonts}
\usepackage{amssymb}
\usepackage{amsthm}
\usepackage{setspace}
\usepackage{enumitem}
\usepackage{nicefrac}
\usepackage{color}

\def\D{{\mathcal D}}
\def\A{{\mathcal A}}

\def\R{{\mathbb R}}

\def \<{\langle}
\def\>{\rangle}

\def \e{\epsilon}

\def \ff{\varphi}

\def \sp{\operatorname{span}}
\def \csp{\overline{\operatorname{span}}}

\def \sign{\operatorname{sign}}

\def\la{\lambda}
\def\a{\alpha}

\newcommand{\n}[1]{\left\| #1 \right\|}	
\newcommand{\pr}[1]{\left( #1 \right)}		
\DeclareMathOperator{\argmin}{argmin}		

\newtheorem{Theorem}{Theorem}[section]
\newtheorem{Lemma}[Theorem]{Lemma}
\newtheorem*{Definition}{Definition}
\newtheorem{Proposition}{Proposition}[section]
\newtheorem*{Remark}{Remark}

\newtheorem{Corollary}{Corollary}[section]
\numberwithin{equation}{section}

\newcommand{\be}{\begin{equation}}
\newcommand{\ee}{\end{equation}}

\begin{document}

\begin{frontmatter}

\title{{A unified way of analyzing some greedy algorithms}}

\author[usc]{A.\,V.~Dereventsov}
\ead{dereventsov@gmail.com}

\author[usc,sim,msu]{V.\,N.~Temlyakov}
\ead{temlyakovv@gmail.com}

\address[usc]{Department of Mathematics, University of South Carolina, Columbia, SC 29208, USA}
\address[sim]{Steklov Institute of Mathematics, Moscow, 119991, Russian Federation}
\address[msu]{Lomonosov Moscow State University, Moscow, 119991, Russian Federation}

\begin{abstract}
In this paper we propose a unified way of analyzing a certain kind of greedy-type algorithms in Banach spaces.
We define a class of the Weak Biorthogonal Greedy Algorithms that contains a wide range of greedy algorithms.
In particular, we show that the following well-known algorithms~--- the Weak Chebyshev Greedy Algorithm and the Weak Greedy Algorithm with Free Relaxation~--- belong to this class.
We investigate the properties of convergence, rate of convergence, and numerical stability of the Weak Biorthogonal Greedy Algorithms.
Numerical stability is understood in the sense that the steps of the algorithm are allowed to be performed with controlled computational inaccuracies.
We carry out a thorough analysis of the connection between the magnitude of those inaccuracies and the convergence properties of the algorithm.
To emphasize the advantage of the proposed approach, we introduce here a new greedy algorithm~--- the Rescaled Weak Relaxed Greedy Algorithm~--- from the above class, and derive the convergence results without analyzing the algorithm explicitly.
Additionally, we explain how the proposed approach can be extended to some other types of greedy algorithms.
\end{abstract}


\end{frontmatter}

\section{Introduction}\label{sec:introduction}

The theory of greedy approximation continues to actively develop, which is driven not only by a theoretical interest but by an applied prospect as well.
As a result, many known greedy-type algorithms are being examined and new ones are being invented to fit specific applications.
The tools employed in the analysis of these algorithms, however, are mostly identical, which is due to the fact that the behavior of a greedy algorithm is largely dictated by the geometry of the space rather than the particular details in the algorithm's description.

The main goal of this paper is to establish and investigate those underlying geometrical connections.
In this effort we present a unified way of analyzing a certain kind of greedy-type algorithms in Banach spaces.
Namely, we define a class of Weak Biorthogonal Greedy Algorithms (denoted further as $\mathcal{WBGA}$) and study the convergence properties for algorithms from this class.
In particular, well-known algorithms like the Weak Chebyshev Greedy Algorithm and the Weak Greedy Algorithm with Free Relaxation belong to this class.

Motivated by the issue of numerical stability, we also discuss modifications of the class $\mathcal{WBGA}$, which allow for the steps of the algorithms to be performed with imprecise calculations.
In accordance with the historical development, we call such algorithms {\it approximate} greedy algorithms.
We show that each of the considered approximate greedy algorithms belongs to a larger class, which we introduce and call the Approximate Weak Biorthogonal Greedy Algorithms (denoted further as $\mathcal{AWBGA}$).
We examine the relations between the $\mathcal{AWBGA}$ and various computational errors that can occur in practice, and emphasize the robustness of the algorithms from this class.

To the best of our knowledge, this paper is the first successful attempt to establish a unified way of analyzing a wide class of greedy-type algorithms.
The step from the analysis of a specific greedy algorithm to a unified analysis of a rather wide class of algorithms motivated us to introduce here a new greedy-type algorithm~--- the Rescaled Weak Relaxed Greedy Algorithm~--- from the above class, and derive the corresponding convergence results without analyzing the algorithm explicitly.

This paper is organized as follows.
In Section~\ref{sec:overview} we describe the problem setting, recall basic definitions and concepts, and introduce and state simple conclusions of our main results.
Section~\ref{sec:dual_greedy} contains a discussion on greedy algorithms and historical comments related to the topic.
Specifically, in this section we state the algorithms of interest and present a new greedy algorithm.
In Section~\ref{sec:wbga} we define the class $\mathcal{WBGA}$ and show that the algorithms of interest belong to that class.
Theoretical results related to the convergence properties of algorithms from the $\mathcal{WBGA}$ are stated in Section~\ref{sec:wbga_convergence} and proven in Section~\ref{sec:wbga_proofs}.
Section~\ref{sec:awbga} is devoted to the class $\mathcal{AWBGA}$ (approximate versions of algorithms from the $\mathcal{WBGA}$).
In Section~\ref{sec:awbga_convergence} we state the convergence results for the $\mathcal{AWBGA}$ and in Section~\ref{sec:awbga_proofs} we prove the stated results.
In Section \ref{sec:x-greedy} we comment on another type of greedy algorithms and explain how our findings can be adapted to fit those algorithms.
In Section~\ref{sec:discussion} we explain and discuss our motivation for the study of the above class of algorithms.

\section{Overview}\label{sec:overview}

The purpose of a greedy algorithm is to iteratively construct an approximation of an element $f \in X$ (the target element) by a linear combination of appropriate elements of a given set $\D \subset X$ (the dictionary).
The exact method of selecting the elements $\{\varphi_m\}_{m=1}^\infty$ from the dictionary and constructing approximators $\{G_m\}_{m=1}^\infty$ depends on the algorithm.
Generally, an $m$-th iteration of a greedy algorithm consists of the following stages:
\begin{itemize}
	\item Select a suitable element $\varphi_m$ from the dictionary $\D$;
	\item Construct an approximation $G_m$ for the target element $f$;
	\item Update the remainder $f_m := f - G_m$, and proceed to the iteration $m+1$.
\end{itemize}

This paper is focused on greedy algorithms in a Banach space setting.
Let $X$ be a real Banach space with a norm $\|\cdot\|$.
We say that a set of elements (functions) $\D$ from $X$ is a dictionary if each $g\in \D$ has the norm bounded by one ($\|g\|\le1$) and $\csp \D = X$.
For convenience we assume that every dictionary $\D$ is symmetric, i.e.
\[
	g \in \D \quad \text{implies} \quad -g \in \D.
\]
By $\A_1(\D)$ we denote the closure of the convex hull of $\D$ in $X$.

In a real Hilbert space it is natural to use the inner product to select the element $\varphi_m$ from the dictionary $\D$, e.g.
\[
	\varphi_m \ \ \text{is such that}\ \ 
	\<f_{m-1}, \varphi_m\> = \sup_{g \in \D} \<f_{m-1}, g\>.
\]
While in a general Banach space one does not have an inner product, there are alternative tools which can be utilized.
One known approach is to select an element $\varphi_m$ by solving an optimization problem, e.g.
\[
	\varphi_m \ \ \text{is such that} \ \ 
	\inf_{\lambda \in \mathbb{R}} \|f_{m-1} - \lambda \varphi_m\|
	= \inf_{g \in \D, \lambda \in \mathbb{R}} \|f_{m-1} - \lambda g\|.
\]
The algorithms that select elements of the dictionary in such a way are called the {\it X-greedy} algorithms.
\\
Another approach for selecting $\varphi_m$ is to use the norming functionals (defined below) to act as an analogue of an inner product, e.g.
\be\label{eq:dual_greedy_gs}
	\varphi_m \ \ \text{is such that}\ \ 
	F_{f_{m-1}}(\varphi_m) = \sup_{g \in \D} F_{f_{m-1}}(g).
\ee
The algorithms that employ norming functionals are called the {\it dual} greedy algorithms.

In this paper we concentrate on the dual greedy algorithms.
In Section~\ref{sec:x-greedy} we briefly discuss how our findings can be applied to analyze the X-greedy algorithms.
For a more detailed discussion on various types of greedy algorithms in the settings of Hilbert and Banach spaces, we refer the reader to the book~\cite{Tbook} and the references therein.

For a nonzero element $f \in X$ we let $F_f$ denote a norming (peak) functional for $f$:
\[
	\|F_f\| = 1, \qquad F_f(f) = \|f\|.
\]
The existence of such a functional is provided by the Hahn--Banach theorem, however the uniqueness is generally not guaranteed.
A Banach space in which each element $f \in X \setminus \{0\}$ has a unique norming functional is the smooth space, i.e. the space with the Gateaux-differentiable norm.
In such a case the value of the norming functional of an element $f$ on an element $g$ is the Gateaux-derivative of the norm $\|\cdot\|$ at $f$ in the direction $g$, i.e.
\[
	F_f(g) = \lim_{u \to 0} \frac{\|f + ug\| - \|f\|}{u}.
\]
Note that in many spaces of interest the norming functionals are known explicitly.
For instance, in $L_p(\Omega)$, $1 \le p < \infty$ we have
\[
	F_f(g) = \int\limits_\Omega \frac{\sign(f) |f|^{p-1} \, g}{\|f\|_p^{p-1}} \, \mathrm{d} \mu.
\]
For convenience of presentation, we introduce the norm $\|\cdot\|_\D$ in the dual space $X'$, associated with the dictionary $\D$ and given by the formula
\[
	\|F\|_\D := \sup_{g \in \D} F(g), \quad F \in X'.
\]
For more information on norming functionals and their connection to smoothness and other geometrical aspects of the space we refer the reader to the book~\cite{beauzamy1982introduction}.

Since formally the selection step~\eqref{eq:dual_greedy_gs} of a dual greedy algorithm might not always be attainable, it has became standard to work with {\it weak} greedy algorithms (proposed in~\cite{T13}), which perform the greedy selection step not precisely but with a controlled relaxation given by a sequence $\tau = \{t_m\}_{m=1}^\infty$, $t_m \in [0,1]$ (the weakness sequence), e.g.
\[
	\varphi_m \ \ \text{is such that}\ \ 
	F_{f_{m-1}}(\varphi_m)
	\ge t_m \sup_{g \in \D} F_{f_{m-1}}(g)
	= t_m \|F_{f_{m-1}}\|_\D.
\]

We now can define the class of Weak Biorthogonal Greedy Algorithms, which contains a large class of weak dual greedy algorithms.
We say that an algorithm belongs to the class $\mathcal{WBGA}$ if it satisfies the following criteria for every $m \ge 1$, where we set $f_0 := f$ and $G_0 := 0$:
\begin{enumerate}[label=\bf(\arabic*), leftmargin=.5in]
	\item{\bf Greedy selection.} At the $m$-th iteration the algorithm selects an element of the dictionary $\varphi_m \in \D$, which satisfies
		\[
			F_{f_{m-1}}(\varphi_m) \ge t_m \|F_{f_{m-1}}\|_\D;
		\]
	\item{\bf Error reduction.} The remainder $f_m$ is changed in such a way that
		\[
			\|f_m\| \le \inf_{\lambda \ge 0} \|f_{m-1} - \lambda \varphi_m\|;
		\]
	\item{\bf Biorthogonality.} The remainder $f_m$ is biorthogonal to the constructed approximant $G_m \in \sp\{\varphi_1,\dots,\varphi_m\}$
		\[
			F_{f_m}(G_m) = 0.
		\]
\end{enumerate}

Note that the $\mathcal{WBGA}$ is not an algorithm but rather the class (collection) of algorithms that possess the properties~\ref{wbga_gs}--\ref{wbga_bo}.
In particular, in Section~\ref{sec:dual_greedy} we state some well-known dual greedy algorithms that belong to this class and introduce a new one.
In Section~\ref{sec:wbga} we discuss some aspects of the class $\mathcal{WBGA}$ and show that the stated algorithms belong to this class.

We prove convergence and rate of convergence results for the $\mathcal{WBGA}$, which provide the above mentioned unified way of analyzing greedy-type algorithms from this class.
Our results on convergence and rate of convergence are formulated in terms of modulus of smoothness of the space.
For a Banach space $X$ the modulus of smoothness $\rho$ is defined as
\be\label{def:modulus_of_smoothness}
	\rho(u) := \rho(u,X) = \sup_{\|x\| = \|y\| = 1} \frac{\|x + uy\| + \|x - uy\|}{2} - 1.
\ee
The uniformly smooth Banach space is the one with the property
\[
	\lim_{u \to 0} \rho(u)/u = 0.
\]
We say that the modulus of smoothness $\rho(u)$ is of power type $1 < q \le 2$ if $\rho(u) \le \gamma u^q$ with some $\gamma > 0$.
In particular, it is known that $L_p$-spaces with $p \in [1,\infty)$ have modulus of smoothness of power type $\min\{p,2\}$ (see e.g.~\cite[Lemma~B.1]{donahue1997rates}).

The main results of the paper~--- Theorems~\ref{thm:wbga_convergence} and~~\ref{thm:wbga_convergence_rate}~--- are stated in Section~\ref{sec:wbga_convergence}.
Below we give a demonstrative corollary of those theorems.
\begin{Corollary}\label{cor:wbga_t_m=t}
Let $X$ be a uniformly smooth Banach space with modulus of smoothness $\rho(u)$ of power type $1 < q \le 2$, that is $\rho(u) \le \gamma u^q$.
Then an algorithm from the $\mathcal{WBGA}$ with a constant weakness sequence $\tau = \{t\}$, $t \in (0,1]$ converges for any target element $f \in X$.
\\
Moreover, if $f \in \A_1(\D)$, then we have
\[
	\|f_m\| \le C(q,\gamma,t) \, m^{-1/p},
\]
where $C(q,\gamma,t) = 16 \gamma^{1/q} \, t^{-1/p}$ and $p = q/(q-1)$.
\end{Corollary}

Corollary~\ref{cor:wbga_t_m=t} addresses two important characteristics of an algorithm~--- convergence and rate of convergence.
We also consider here another essential characteristic~--- computational stability, which is crucial for practical implementation.
A systematic study of the stability of greedy algorithms in Banach spaces was started in~\cite{T7}, where sufficient conditions for convergence and rate of convergence were established.
A further development of the theory of stability of greedy algorithms was conducted in~\cite{De}, where necessary and sufficient conditions for the convergence of a single algorithm were obtained.

In Section~\ref{sec:awbga} we address the stability issue of algorithms from the class $\mathcal{WBGA}$.
We analyze an extended version of the $\mathcal{WBGA}$, which allows the steps of an algorithm to be performed with some controlled inaccuracies.
Algorithms with imprecise step evaluations are well-known in the greedy approximation theory; this type of algorithms are called the {\it approximate} greedy algorithms (see e.g.~\cite{gribonval2001approximate, galatenko2003convergence, galatenko2005generalized}).
Permissions of approximate evaluations are natural for the applications as they simplify and speed up the execution of an algorithm.

We begin by defining the variables that represent these inaccuracies (called the error parameters or the error sequences).
Let $\{\delta_m\}_{m=0}^\infty$, $\{\eta_m\}_{m=1}^\infty$ be  sequences of real numbers from the interval $[0,1]$, and let $\{\e_m\}_{m=1}^\infty$ be a sequence of non-negative numbers. 
Denote by $\{F_m\}_{m=0}^\infty$ a sequence of such functionals that for any $m \geq 0$
\[
	\|F_m\| \leq 1
	\ \text{ and }\ 
	F_m(f_m) \geq (1-\delta_m)\|f_m\|,
\]
where $\{f_m\}_{m=0}^\infty$ is the sequence of remainders produced by the algorithm.

We say that an algorithm belongs to the class $\mathcal{AWBGA}$ if for every $m \ge 1$ it satisfies the following relaxed versions of conditions~\ref{wbga_gs}--\ref{wbga_bo}:
\begin{enumerate}[label=\bf(\arabic*), leftmargin=.5in]
	\item{\bf Greedy selection.} At the $m$-th iteration the algorithm selects an element of the dictionary $\varphi_m \in \D$, which satisfies
		\[
			F_{m-1}(\varphi_m) \ge t_m \|F_{m-1}\|_\D;
		\]
	\item{\bf Error reduction.} The remainder $f_m$ is changed in such a way that
		\[
			\|f_m\| \le (1 + \eta_m) \inf_{\lambda \ge 0} \|f_{m-1} - \lambda \varphi_m\|;
		\]
	\item{\bf Biorthogonality.} The constructed approximant $G_m \in \sp\{\varphi_1,\dots,\varphi_m\}$ is such that
		\[
			|F_m(G_m)| \le \e_m.
		\]
\end{enumerate}

Similarly to the $\mathcal{WBGA}$, in Section~\ref{sec:awbga_convergence} we state convergence and rate of convergence results for the $\mathcal{AWBGA}$ (Theorems~\ref{thm:awbga_convergence} and~\ref{thm:awbga_convergence_rate}).
For illustrative purposes we formulate here a simple corollary of those results.
\begin{Corollary}\label{cor:awbga_t_m=t}
Let $X$ be a uniformly smooth Banach space with the modulus of smoothness $\rho(u)$ of power type $1 < q \le 2$.
Then an algorithm from the $\mathcal{AWBGA}$ with a constant weakness sequence $\tau = \{t\}$, $t \in (0,1]$ and error parameters $\{\delta_m\}_{m=0}^\infty$, $\{\eta_m\}_{m=1}^\infty$, $\{\e_m\}_{m=1}^\infty$ with
\[
	\lim_{m\to\infty} \delta_m
	= \lim_{m\to\infty} \eta_m
	= \lim_{m\to\infty} \e_m
	= 0
\]
converges for any target element $f \in X$.
\\
Moreover, if $f \in \A_1(\D)$ and
\[
	\delta_m + \e_m/\|f_m\| \leq 1/4,
	\quad
	\delta_m + \eta_{m+1} \leq C(q,\gamma,t)^{-p} \|f_m\|^p,
\]
then we get
\[
	\|f_m\| \le C(q,\gamma,t) \, m^{-1/p},
\]
where $C(q,\gamma,t) = 16 \gamma^{1/q} \, t^{-1/p}$ and $p = q/(q-1)$.
\end{Corollary}

\section{Dual greedy algorithms}\label{sec:dual_greedy}

The goal of this paper is to establish a unified way of analyzing a certain kind of dual greedy algorithms.
In this section we briefly overview the most prominent algorithms and explicitly define the ones that are directly related to our new results.
By $\tau := \{t_m\}_{m=1}^\infty$ we denote a weakness sequence, i.e. a given sequence of non-negative numbers $t_m \le 1$, $m=1,2,3,\dots$.

We define first the Weak Chebyshev Greedy Algorithm (WCGA) (introduced in~\cite{T15}) that can be thought of as a generalization of the Orthogonal Greedy Algorithm (defined and studied in~\cite{DTe}, also known as the Orthogonal Matching Pursuit~\cite{pati1993orthogonal}) to a Banach space setting.

\bigskip\noindent
{\bf Weak Chebyshev Greedy Algorithm (WCGA)\\}
Set $f^c_0 := f$ and $G^c_0 := 0$.
Then for each $m \ge 1$
\begin{enumerate}
	\item Take any $\varphi^{c}_m \in \D$ satisfying $F_{f^{c}_{m-1}}(\varphi^{c}_m) \ge t_m \| F_{f^{c}_{m-1}}\|_\D$;
	\item Denote $\Phi_m^c := \sp \{\varphi^{c}_j\}_{j=1}^m$ and find $G_m^c \in \Phi_m^c$ such that
		\[
			G_m^c = \mathop{\argmin}_{G \in \Phi_m^c} \n{f - G};
		\]
	\item Set $f^{c}_m := f-G^c_m$.
\end{enumerate}
\smallskip

Note that in order to construct the $m$-th approximant $G_m^c$ the WCGA has to solve an $m$-dimensional optimization problem
\[
	(\lambda_1^m, \ldots, \lambda_m^m)
	= \mathop{\argmin}_{(\lambda_1, \ldots, \lambda_m) \in \mathbb{R}^m} \n{f - \sum_{k=1}^m \lambda_k \varphi_k^c},
	\quad
	G_m^c = \sum_{k=1}^m \lambda_k^m \varphi_k^c.
\]
From a computational perspective, such a problem, even though convex, becomes more expensive with each iteration.

The greedy algorithm that attempts to maximally simplify the approximant construction process is the Weak Dual Greedy Algorithm (WDGA)~\cite{dilworth2002convergence}, which can be thought of as a generalization of the Pure Greedy Algorithm (\cite{DTe}, also known as the Matching Pursuit, see e.g.~\cite{mallat1993matching}).
In the WDGA the approximation $G_m$ is constructed by solving a one-dimensional convex optimization problem
\[
	\lambda_m = \mathop{\argmin}_{\lambda \ge 0} \n{f_{m-1} - \lambda \varphi_m},
	\quad
	G_m = G_{m-1} + \lambda_m \varphi_m.
\]
Such an approach is much simpler computationally when compared to the full $m$-dimensional projection performed by the WCGA, however the convergence results for the WDGA are quite scarce (see~\cite{dilworth2002convergence, ganichev2003convergence}).

There is a famous technique of constructing the next approximation as a linear combination of the previous approximant and a newly chosen element of the dictionary, i.e.
\[
	G_m = \omega_m G_{m-1} + \lambda_m \varphi_m
\]
with appropriately chosen coefficients $\omega_m$ and $\lambda_m$.
Greedy algorithms that employ such an approach are called the {\it relaxed} greedy algorithms.
While the parameter $\lambda_m$ is generally selected by the line search, the selection of $\omega_m$ is where relaxed greedy algorithms differ.

A particular approach in this direction is to omit the selection of $\omega_m$ by taking it to be the convex complement of $\lambda_m$, i.e. $\omega_m = 1 - \lambda_m$.
Such a technique is in the style of the renown Frank--Wolfe algorithm (see e.g.~\cite{frank1956algorithm, jaggi2013revisiting}).
The greedy algorithm that utilizes this concept is the Weak Relaxed Greedy Algorithm~\cite{T15}, which constructs an approximant in the following way
\[
	\lambda_m = \mathop{\argmin}_{0 \le \lambda \le 1} \|f - ((1 - \lambda) G_{m-1} + \lambda \varphi_m)\|,
	\quad
	G_m = (1 - \lambda_m) G_{m-1} + \lambda_m \varphi_m.
\]
Despite the attractive computational simplicity of this approach, it is naturally limited in the sense that the algorithm converges only for the target elements from the closure of the convex hull of the elements of the dictionary (i.e. the class $\A_1(\D)$).

An easy way to mitigate this restriction is by not imposing any restraining connections on coefficients $\omega_m$ and $\lambda_m$.
Perhaps the most straightforward realization of this concept, which is closer to the style of the WCGA, is to solve the minimization problem over both parameters $\omega_m,\lambda_m$.
This is the idea behind the Weak Greedy Algorithm with Free Relaxation~\cite{T26}.

\bigskip\noindent
{\bf Weak Greedy Algorithm with Free Relaxation (WGAFR)\\}
Set $f^e_0 := f$ and $G^e_0 := 0$.
Then for each $m \ge 1$
\begin{enumerate}
	\item Take any $\varphi^e_m \in \D$ satisfying $F_{f^e_{m-1}}(\varphi^e_m) \ge t_m \|F_{f^e_{m-1}}\|_\D$;
	\item Find $\omega_m \in \mathbb{R}$ and $ \lambda_m \ge 0$ such that
		\[
			\|f - ((1 - \omega_m) G^e_{m-1} + \la_m \varphi^e_m)\|
			= \inf_{ \la \ge 0, \omega \in \mathbb{R}} \|f - ((1 - \omega) G^e_{m-1} + \la \varphi^e_m)\|
		\]
		and define $G^e_m := (1 - \omega_m) G^e_{m-1} + \la_m \varphi^e_m$;
	\item Set $f^e_m := f - G^e_m$.
\end{enumerate}
\smallskip

Note that the second step of the WGAFR is generally much simpler computationally than the second step the WCGA since at the $m$-th iteration the WGAFR solves a $2$-dimensional optimization problem, while the WCGA has to solve an $m$-dimensional one.

Nevertheless, in some problems it might be preferable to simplify the algorithm even further by decoupling the choice of parameters $\omega_m,\lambda_m$ and hence reducing the dimensionality of the minimization problem.
Based on this idea, we propose here a new relaxed dual greedy algorithm that employs two one-dimensional minimizations~--- the Rescaled Weak Relaxed Greedy Algorithm (RWRGA).

\bigskip\noindent
{\bf Rescaled Weak Relaxed Greedy Algorithm (RWRGA)\\}
Set $f^r_0 := f$ and $G^r_0 := 0$.
Then for each $m \ge 1$
\begin{enumerate}
	\item Select any $\varphi^r_m \in \D$ satisfying $F_{f^r_{m-1}}(\varphi^r_m) \ge t_m \|F_{f^r_{m-1}}\|_\D$;
	\item Find $\lambda_m \ge 0$ such that
		\[
			\|f^r_{m-1} - \la_m \varphi^r_m\| = \inf_{\la \ge 0} \|f^r_{m-1} - \la \varphi^r_m\|;
		\]
	\item Find $\mu_m \in \mathbb{R}$ such that
		\[
			\|f - \mu_m (G^r_{m-1} + \la_m \varphi^r_m)\| 
			= \inf_{\mu \in \mathbb{R}} \|f - \mu (G^r_{m-1} + \la_m \varphi^r_m)\|
		\]
		and define $G^r_m := \mu_m (G^r_{m-1} + \la_m \varphi^r_m)$;
	\item Set $f^r_m := f - G^r_m$.
\end{enumerate}
\smallskip

Note that the idea of selecting the parameters $\omega_m$ and $\lambda_m$ independently is quite popular and has been considered before in various areas.
In particular, the approach of choosing $\{\omega_m\}_{m=1}^\infty$ as values of a predefined sequence is considered in~\cite{BCDD} for a Hilbert space, and in~\cite{T26} for a general Banach space setting (see also~\cite[Chapter~6.5]{Tbook}).

Recently, an interesting new idea of building relaxed greedy algorithms was proposed in~\cite{GG}, where the $\la_m$ is specified explicitly in terms of $f_{m-1}$ and of parameters $\gamma$ and $q$ characterizing modulus of smoothness~\eqref{def:modulus_of_smoothness} of the Banach space $X$:
\[
	\la_m := \sign(F_{f_{m-1}}(\ff_m)) \, \|f_{m-1}\| \, \pr{\frac{|F_{f_{m-1}}(\ff_m)|}{2\gamma q}}^{1/(q-1)},
\]
and the scaling parameter $\mu_m$ is chosen via the line search.

Such an algorithm is even simpler computationally than the RWRGA as it involves only one $1$-dimensional optimization problem rather than two.
However, this kind of further simplification might not always be beneficial as it results in certain drawbacks.
For instance, one requires an {\it a priori} knowledge about the Banach space in terms of the values of $\gamma$ and $q$.
Moreover, since the value of $\lambda_m$ directly depends on the evaluation of $F_{f_{m-1}}$, it is hard to predict how the algorithm behaves under the imprecise evaluation of the norming functionals.
Surely, this algorithm is beneficial in certain settings when the smoothness of the space is known (e.g. $L_p$-spaces) and the precision of computations is not an issue; however the detailed comparison of various greedy algorithms is outside of the scope of this paper.

In the next section we show that despite the differences in the approximant construction and computational complexity, all three algorithms (the WCGA, the WGAFR, and the RWRGA) admit the same rate of convergence estimates (at least theoretically) due to the fact that the stages of those algorithms (selection of the element $\varphi_m$ and construction of the approximant $G_m$) possess the same geometrical properties~\ref{wbga_gs}--\ref{wbga_bo}.

\section{Weak Biorthogonal Greedy Algorithms}\label{sec:wbga}

While in the previous section we introduced the variety of seemingly different dual greedy algorithms, in this section we show that the convergence analysis of those algorithms is based on underlying geometrical properties which are identical for all algorithms.

This is our motivation to establish those properties by defining the class of the Weak Biorthogonal Greedy Algorithms (denoted further as $\mathcal{WBGA}$), introduced in Section~\ref{sec:overview}.

We say that an algorithm belongs to the class $\mathcal{WBGA}$ if sequences of remainders $\{f_m\}_{m=0}^\infty$, approximators $\{G_m\}_{m=0}^\infty$, and selected elements $\{\varphi_m\}_{m=1}^\infty$ satisfy the following conditions at every iteration $m \ge 1$:
\begin{enumerate}[label=\bf(\arabic*), leftmargin=.5in]
	\item\label{wbga_gs}
		Greedy selection: ${\displaystyle F_{f_{m-1}}(\varphi_m  ) \ge t_m \|F_{f_{m-1}}\|_\D}$;
	\item\label{wbga_er}
		Error reduction: ${\displaystyle \|f_m\| \le \inf_{\la\ge 0}\|f_{m-1} -\la \ff_m\|}$;
	\item\label{wbga_bo}
		Biorthogonality: ${\displaystyle F_{f_m}(G_m) = 0}$.
\end{enumerate}

\begin{Remark}
Note that it is sufficient to take the infimum over $\lambda \ge 0$ in the error reduction condition~\ref{wbga_er} since the greedy selection condition~\ref{wbga_gs} implies that
\[
	\inf_{\lambda \in \mathbb{R}} \|f_{m-1} - \lambda \varphi_m\| = \inf_{\la \ge 0} \|f_{m-1} - \la \varphi_m\|.
\]
Indeed, for any $\lambda < 0$ we have
\begin{align*}
	\|f_{m-1} - \lambda \varphi_m\|
	&\ge F_{f_{m-1}} (f_{m-1} - \lambda \varphi_m)
	\\
	&\ge \|f_{m-1}\| - \lambda t_m \|F_{f_{m-1}}\|_\D
	\ge \|f_{m-1}\|.
\end{align*}
\end{Remark}

\begin{Remark}
Note that the biorthogonality condition~\ref{wbga_bo} guarantees that the algorithm fully utilizes the constructed approximant $G_m$ to reduce the norm of remainder $f_m$.
Indeed, for any $\lambda \in \mathbb{R}$ we have
\[
	\|f_m - \lambda G_m\|
	\ge F_{f_m}(f_m - \lambda G_m)
	= \|f_m\|.
\]
The above estimate implies that the remainder $f_m$ is orthogonal to the approximant $G_m$ in the sense of Birkhoff--James orthogonality (see e.g.~\cite{birkhoff1935orthogonality}).
\end{Remark}

We now prove that the algorithms from Section~\ref{sec:dual_greedy}~--- the WCGA, the WGAFR, and the RWRGA~--- belong to the class $\mathcal{WBGA}$.
\begin{Proposition}\label{prp:dga_wbga}
The WCGA, the WGAFR, and the RWRGA belong to the class $\mathcal{WBGA}$.
\end{Proposition}

We will use the following well-known result.
\begin{Lemma}[{\cite[Lemma~6.9]{Tbook}}]\label{lem:orthogonality}
Let $X$ be a uniformly smooth Banach space and $L$ be a finite-dimensional subspace of $X$.
For any $f \in X \setminus L$ let $f_L$ denote the best approximant of $f$ from $L$.
Then for any $\phi \in L$ we have 
\[
	F_{f-f_L}(\phi) = 0.
\]
\end{Lemma}

\begin{proof}[Proof of Proposition~\ref{prp:dga_wbga}]
From the definitions of the corresponding algorithms it is clear that the greedy selection condition~\ref{wbga_gs} holds for all three algorithms.

For the WCGA / WGAFR the error reduction condition~\ref{wbga_er} follows from the fact that $G^c_m$ / $G^e_m$ is the best approximation to $f$ from the subspace $\sp\{\varphi^c_1,\ldots,\varphi^c_m\}$ / $\sp\{G^e_{m-1},\varphi^e_m\}$ respectively.
Evidently, the biorthogonality condition~\ref{wbga_bo} is guaranteed by the previous observation and Lemma~\ref{lem:orthogonality}.
\\
For the RWRGA we write
\begin{align*}
	\|f^r_m\|
	&= \inf_{\mu \in \mathbb{R}} \|f - \mu (G^r_{m-1} + \la_m \varphi^r_m)\| 
	\\
	&\le \|f - G^r_{m-1} - \lambda_m \varphi^r_m\|
	= \inf_{\lambda \ge 0} \|f_{m-1} - \lambda \varphi^r_m\|.
\end{align*}
Hence condition~\ref{wbga_er} holds and condition~\ref{wbga_bo} follows from Lemma~\ref{lem:orthogonality} and the observation that $G^r_m$ is the best approximation from $\sp\{G^r_{m-1} + \la_m \varphi^r_m\}$.
\end{proof}

We emphasize the advantage of analyzing the class $\mathcal{WBGA}$ by noting that any result proven for the class holds for each algorithm in the class.
For instance, it allows us to derive theoretical estimates on the behavior of the RWRGA, which is a novel algorithm and has not been studied previously.
In the next section we claim and prove the statements regarding the convergence and the rate of convergence for the $\mathcal{WBGA}$ and compare them to the known results for some algorithms.

\section{Convergence results for the $\mathcal{WBGA}$}\label{sec:wbga_convergence}

In the formulation of the convergence results for the $\mathcal{WBGA}$ we utilize a special sequence which is defined for a given modulus of smoothness $\rho(u)$ and a given weakness sequence $\tau = \{t_m\}_{m=1}^\infty$.
\begin{Definition}
For a given weakness sequence $\tau$ and a number $0 < \theta \le 1/2$ we define $\xi_m := \xi_m(\rho,\tau,\theta)$ as a root of the equation
\be\label{def:xi}
	\rho(u) = \theta t_m u.
\ee
\end{Definition}

\begin{Remark}
Note that a modulus of smoothness $\rho(u)$, defined in~\eqref{def:modulus_of_smoothness}, is an even convex function on $(-\infty,\infty)$ and $\rho(2) \ge 1$.
Hence the function
\[
	s(u) := 
		\left\{\begin{array}{ll}
			\rho(u)/u, & u \neq 0
			\\
			0, & u = 0
		\end{array}\right.
\]
is continuous and increasing on $[0,\infty)$ with $s(2) \ge 1/2$.
Thus the equation~\eqref{def:xi} has a unique solution $\xi_m = s^{-1}(\theta t_m)$ such that $0 < \xi_m \le 2$.  
\end{Remark}

We now state our main convergence result for the $\mathcal{WBGA}$.
\begin{Theorem}\label{thm:wbga_convergence}
Let $X$ be a uniformly smooth Banach space with the modulus of smoothness $\rho(u)$.
Let $\tau = \{t_m\}_{m=1}^\infty$ be such a sequence that for any $\theta > 0$ we have
\be\label{eq:sum_t_xi}
	\sum_{m=1}^\infty t_m \xi_m(\rho,\tau,\theta) = \infty.
\ee
Then an algorithm from the $\mathcal{WBGA}$ with the weakness sequence $\tau$ converges for any target element $f \in X$.
\end{Theorem}

Condition~\eqref{eq:sum_t_xi} essentially holds for weakness sequences $\tau$ that do not decay rapidly.
We demonstrate this idea with with the following corollary. 
\begin{Corollary}\label{cor:wbga_convergence}
Let $X$ be a uniformly smooth Banach space with the modulus of smoothness $\rho(u)$ of power type $1 < q \le 2$.
Assume that 
\be\label{eq:sum_t_m^p}
	\sum_{m=1}^\infty t_m^p = \infty,
	\quad
	p = \frac{q}{q-1}. 
\ee
Then an algorithm from the $\mathcal{WBGA}$ with the weakness sequence $\tau$ converges for any target element $f \in X$.
\end{Corollary}

Another important result of the paper stated below describes the rate of convergence of the $\mathcal{WBGA}$. 
\begin{Theorem}\label{thm:wbga_convergence_rate}
Let $X$ be a uniformly smooth Banach space with the modulus of smoothness $\rho(u)$ of power type $1 < q \le 2$.
Let $f \in X$ be a target element.
Take a number $\e \ge 0$ and an $f^\e$ from $X$ such that
\[
	\|f - f^\e\| \le \e,
	\quad
	f^\e / A(\e) \in \A_1(\D)
\]
with some number $A(\e) > 0$.
Then for an algorithm from the $\mathcal{WBGA}$ we have
\[
	\|f_m\| \le \max \left\{ 2\e,\, C(q,\gamma) (A(\e) + \e) \Big(1 + \sum_{k=1}^m t_k^p \Big)^{-1/p} \right\},
\]
where $C(q,\gamma)= 4(2\gamma)^{1/q}$ and $p = q/(q-1)$.
\end{Theorem}

This result indicates a pragmatic advantage of the class $\mathcal{WBGA}$: it does not use any {\it a priori} information about the target element $f$ or the space $X$ and automatically adjusts to its smoothness property.
Moreover, the above formulation of Theorem~\ref{thm:wbga_convergence_rate} covers the case of noisy data since $f$ can be viewed as a noisy version of $f^\e$.

Theorem~\ref{thm:wbga_convergence_rate} shows that the rate of convergence of an algorithm depends on how close the target element is to the closure of the convex hull of the elements of the dictionary, which is a known fact in the area of greedy approximation.
The following corollary gives an estimate on the rate of convergence for a target element from the class $\A_1(\D)$.
\begin{Corollary}\label{cor:wbga_convergence_rate}
Let $X$ be a uniformly smooth Banach space with the modulus of smoothness $\rho(u)$ of power type $1 < q \le 2$.
Then for any target element $f \in \A_1(\D)$ an algorithm from the $\mathcal{WBGA}$ provides
\[
	\|f_m\| \le C(q,\gamma) \Big(1 + \sum_{k=1}^m t_k^p \Big)^{-1/p},
\]
where $C(q,\gamma)= 4(2\gamma)^{1/q}$ and $p = q/(q-1)$.
\end{Corollary}

\begin{Remark}
Theorems~\ref{thm:wbga_convergence} and~\ref{thm:wbga_convergence_rate} are known for the WCGA (see~\cite{T15}) and the WGAFR (see~\cite{T26}).
For the RWRGA stated results are novel.
\end{Remark}

\section{Proofs of results from Section~\ref{sec:wbga_convergence}}\label{sec:wbga_proofs}

First, we prove the convergence results for the $\mathcal{WBGA}$.
We start with the well-known lemma.
\begin{Lemma}[{\cite[Lemma~6.10]{Tbook}}]\label{L2.2c}
For any bounded linear functional $F$ and any dictionary $\D$ we have
\[
	\|F\|_\D = \sup_{g \in \D} F(g) = \sup_{f \in \A_1(\D)} F(f).
\]
\end{Lemma}

The following lemma is our main tool for analyzing the behavior of algorithms from the $\mathcal{WBGA}$.
\begin{Lemma}[{{\bf Error Reduction Lemma}}]\label{lem:wbga_er}
Let $X$ be a uniformly smooth Banach space with the modulus of smoothness $\rho(u)$.
Take a number $\e \ge 0$ and two elements $f$, $f^\e$ from $X$ such that
\[
	\|f - f^\e\| \le \e,
	\quad
	f^\e / A(\e) \in \A_1(\D)
\]
with some number $A(\e) \ge \e$.
Then for an algorithm from the $\mathcal{WBGA}$ we have
\[
	\|f_m\| \le \|f_{m-1}\| \inf_{\la\ge0} \left(1 - \la t_m A(\e)^{-1} 
		\left(1 - \frac{\e}{\|f_{m-1}\|}\right) + 2\rho\left(\frac{\la}{\|f_{m-1}\|}\right)\right)
\]
for any $m \ge 1$.
\end{Lemma}

\begin{proof}[Proof of Lemma~\ref{lem:wbga_er}]
Due to the error reduction property~\ref{wbga_er} it is sufficient to bound the quantity $\inf_{\la\ge 0} \|f_{m-1} - \la \ff_m\|$.
From the definition of modulus of smoothness~\eqref{def:modulus_of_smoothness} we have for any $\la \ge 0$
\[
	\|f_{m-1 } - \la \ff_m\| + \|f_{m-1}+\la\ff_m\| \le 2\|f_{m-1}\|\left(1+\rho(\la/\|f_{m-1}\|)\right).
\]
Next,
\[
	\|f_{m-1} + \la \ff_m\| \ge F_{f_{m-1}}(f_{m-1} + \la \ff_m)
	= \|f_{m-1}\| + \la F_{f_{m-1}}(\ff_m).
\]
By the greedy selection condition~\ref{wbga_gs} and Lemma~\ref{L2.2c} we get
\begin{align*}
	F_{f_{m-1}}(\ff_m)
	&\ge t_m \sup_{g\in \D} F_{f_{m-1}}(g)
	= t_m \sup_{\phi\in \A_1(\D)} F_{f_{m-1}}(\phi)
	\\
	&\ge t_m A(\e)^{-1} F_{f_{m-1}}(f^\e)
	\ge t_m A(\e)^{-1} (F_{f_{m-1}}(f) - \e).
\end{align*}
The biorthoganality property~\ref{wbga_bo} provides 
\[
	F_{f_{m-1}}(f) = F_{f_{m-1}}(f_{m-1} + G_{m-1})
	= F_{f_{m-1}}(f_{m-1})
	= \|f_{m-1}\|.
\]
Combining the provided estimates completes the proof.
\end{proof}

Next, we prove the general convergence result for the $\mathcal{WBGA}$.
\begin{proof}[Proof of Theorem~\ref{thm:wbga_convergence}]
The Error Reduction Lemma~\ref{lem:wbga_er} implies that $\{\|f_m\|\}_{m=0}^\infty$ is a non-increasing sequence.
Therefore, we have
\[
	\lim_{m \to \infty} \|f_m\| = \a.
\]
We prove that $\a = 0$ by contradiction.
Assume that for any $m$ we have
\[
	\|f_m\| \ge \a > 0.
\]
We set $\e = \a/2$ and find $f^\e$ such that
\[
	\|f - f^\e\| \le \e
	\quad\text{and}\quad
	f^\e / A(\e) \in \A_1(\D)
\]
with some $A(\e)$.
Then, by the Error Reduction Lemma~\ref{lem:wbga_er} we get
\[
	\|f_m\| \le \|f_{m-1}\| \inf_{\la \ge 0} (1 - \la t_m A(\e)^{-1}/2 + 2\rho(\la/\a)).
\]
We specify $\theta := \nicefrac{\a}{8A(\e)}$ and take $\la = \a \xi_m(\rho,\tau,\theta)$, where $\xi_m$ is the solution of~\eqref{def:xi}, and obtain
\[
	\|f_m\| \le \|f_{m-1}\| (1 - 2\theta t_m\xi_m).
\]
Then the assumption~\eqref{eq:sum_t_xi}
\[
	\sum_{m=1}^\infty t_m \xi_m = \infty
\]
implies the contradiction
\[
	\|f_m\| \to 0
	\ \ \text{as}\ \ 
	m \to \infty.
\]
\end{proof}

\begin{proof}[Proof of Corollary~\ref{cor:wbga_convergence}]
Denote $\rho^q(u) = \gamma u^q$.
Then
\[
	\rho(u)/u \le \rho^q(u)/u
\]
and therefore for any $\theta > 0$ we have
\[
	\xi_m(\rho,\tau,\theta) \ge \xi_m(\rho^q,\tau,\theta).
\]
For $\rho^q(u)$ we get from the definition of $\xi_m$~\eqref{def:xi} that
\[
	\xi_m(\rho^q,\tau,\theta) = (\theta t_m/\gamma)^{\frac{1}{q-1}}.
\]
Thus condition~\eqref{eq:sum_t_m^p} implies that
\[
	\sum_{m=1}^\infty t_m \xi_m(\rho,\tau,\theta)
	\ge \sum_{m=1}^\infty t_m \xi_m(\rho^q,\tau,\theta)
	\asymp \sum_{m=1}^\infty t_m^p
	= \infty.
\]
It remains to apply Theorem~\ref{thm:wbga_convergence}.
\end{proof}

Finally, we prove the rate of convergence for algorithms from the $\mathcal{WBGA}$.
\begin{proof}[Proof of Theorem~\ref{thm:wbga_convergence_rate}]
Fix an $m \ge 1$.
Note that the required estimate holds trivially if $A(\e) \le \e$ or $\|f_m\| \le 2\e$.
Consider the case $A(\e) > \e$ and $\|f_m\| > 2\e$.
Then due to monotonicity of the norms of remainders $f_k$ we get that $\|f_k\| > 2\e$ for any $k \le m$. 
Hence the Error Reduction Lemma~\ref{lem:wbga_er} provides
\be\label{3.9n}
	\|f_k\| \le \|f_{k-1}\| \inf_{\la\ge0} \left(1 - \frac{\la t_k A(\e)^{-1}}{2} + 2\gamma\left(\frac{\la}{\|f_{k-1}\|}\right)^q\right).
\ee
Take $\la$ to be the root of the equation
\[
	\frac{\la t_k}{4A(\e)} = 2\gamma \left(\frac{\la}{\|f_{k-1}\|}\right)^q
	\ \ \text{i.e.}\ \ 
	\la = \|f_{k-1}\|^{\frac{q}{q-1}} \pr{\frac{t_k}{8\gamma A(\e)}}^{\frac{1}{q-1}}.
\]
Denote $A_q := 4(8\gamma)^{\nicefrac{1}{q-1}}$, then we get from~\eqref{3.9n}
\[
	\|f_k\| \le \|f_{k-1}\| \left(1 - \frac{1}{4} \frac{\la t_k}{A(\e)}\right)
	= \|f_{k-1}\| \left(1 - \frac{t_k^p \|f_{k-1}\|^p}{A_q A(\e)^p}\right).
\]
Raising both sides of this inequality to the power $p$ and taking into account the inequality $x^r\le x$ for $r\ge 1$, $0\le x\le 1$, we get
\[
	\|f_k\|^p \le \|f_{k-1}\|^p \left(1 - \frac{t^p_k\|f_{k-1}\|^p}{A_q A(\e)^p}\right).
\]
Then by using the estimates $\|f\| \le A(\e) + \e$ and $A_q > 1$, we obtain (see e.g.~\cite[Lemma~3.4]{DTe})
\[
	\|f_m\|^p \le A_q(A(\e)+\e)^p \left(1 + \sum_{k=1}^m t_k^p\right)^{-1},
\]
which implies
\[
	\|f_m\|\le C(q,\gamma)(A(\e) + \e) \left(1 + \sum_{k=1}^m t_k^p\right)^{-1/p}
\]
with $C(q,\gamma) = A_q^{1/p}= 4(2\gamma)^{1/q}$.
\end{proof}

\section{Approximate Weak Biorthogonal Greedy Algorithms}\label{sec:awbga}

In this section we address the issue of numerical stability of the algorithms from the class $\mathcal{WBGA}$ by considering a wider class that we call the Approximate Weak Biorthogonal Greedy Algorithms introduced in Section~\ref{sec:overview}.

The algorithms from the $\mathcal{AWBGA}$ are allowed to perform the steps not exactly but with some computational inaccuracies.
Such numerical inaccuracies are controlled by the weakness sequence $\{t_m\}_{m=1}^\infty$, and by the error sequences $\{\delta_m\}_{m=0}^\infty$, $\{\eta_m\}_{m=1}^\infty$, and $\{\e_m\}_{m=1}^\infty$, where the first three sequences consist of numbers from the interval $[0,1]$ and the last one is of non-negative numbers.
By $\{F_m\}_{m=0}^\infty$ we denote a sequence of functionals that satisfy for any $m \geq 0$ the conditions
\[
	\|F_m\| \leq 1
	\ \text{ and }\ 
	F_m(f_m) \geq (1 - \delta_m) \|f_m\|,
\]
where $\{f_m\}_{m=0}^\infty$ is the sequence of remainders produced by the algorithm.

We say that an algorithm belongs to the class $\mathcal{AWBGA}$ if sequences of remainders $\{f_m\}_{m=0}^\infty$, approximators $\{G_m\}_{m=0}^\infty$, and selected elements $\{\varphi_m\}_{m=1}^\infty$ satisfy the following conditions at every iteration $m \ge 1$:
\begin{enumerate}[label=\bf(\arabic*), leftmargin=.5in]
	\item\label{awbga_gs}
		Greedy selection: ${\displaystyle F_{m-1}(\varphi_m  ) \ge t_m \|F_{m-1}\|_\D}$;
	\item\label{awbga_er}
		Error reduction: ${\displaystyle \|f_m\| \le (1 + \eta_m) \inf_{\la\ge 0}\|f_{m-1} -\la \ff_m\|}$;
	\item\label{awbga_bo}
		Biorthogonality: ${\displaystyle |F_m(G_m)| \le \e_m}$.
\end{enumerate}

We now state the approximate versions of the algorithms from Section~\ref{sec:dual_greedy} (namely, the WCGA, the WGAFR, and the RWRGA) and show that they belong to the class $\mathcal{AWBGA}$.

We start with an approximate version of the WCGA, which was studied in~\cite{T7} and~\cite{De}.
A more general version with both relative and absolute inaccuracies was considered in~\cite{dereventsov2017generalized}.

\bigskip\noindent
{\bf Approximate Weak Chebyshev Greedy Algorithm (AWCGA)\\}
Set $f^c_0 := f$ and $G^c_0 := 0$.
Then for each $m \ge 1$ perform the following steps
\begin{enumerate}
	\item Take any $\varphi^{c}_m \in \D$ satisfying $F_{m-1}(\varphi^{c}_m) \ge t_m \| F_{m-1}\|_\D$;
	\item Denote $\Phi_m^c := \sp \{\varphi^{c}_k\}_{k=1}^m$ and find $G_m^c \in \Phi_m^c$ such that
		\[
			\|f - G_m^c\| \le (1 + \eta_m) \inf_{G \in \Phi_m^c} \n{f - G};
		\]
	\item Set $f^{c}_m := f-G^c_m$.
\end{enumerate}
\smallskip

Next, we introduce an approximate version of the WGAFR.

\bigskip\noindent
{\bf Approximate Weak Greedy Algorithm with Free Relaxation\\(AWGAFR)\\}
Set $f^e_0 := f$ and $G^e_0 := 0$.
Then for each $m \ge 1$
\begin{enumerate}
	\item Take any $\varphi^e_m \in \D$ satisfying $F_{m-1}(\varphi^e_m) \ge t_m \|F_{m-1}\|_\D$;
	\item Find $\omega_m \in \mathbb{R}$ and $ \lambda_m \ge 0$ such that
		\[
			\|f - ((1 - \omega_m) G^e_{m-1} + \la_m \varphi^e_m)\|
			\le (1 + \eta_m) \inf_{ \la \ge 0, \omega \in \mathbb{R}} \|f - ((1 - \omega) G^e_{m-1} + \la \varphi^e_m)\|
		\]
		and define $G^e_m := (1 - \omega_m) G^e_{m-1} + \la_m \varphi^e_m$;
	\item Set $f^e_m := f - G^e_m$.
\end{enumerate}
\smallskip

Lastly, we define an approximate version of the RWRGA.

\bigskip\noindent
{\bf Approximate Rescaled Weak Relaxed Greedy Algorithm (ARWRGA)\\}
Set $f^r_0 := f$ and $G^r_0 := 0$.
Then for each $m \ge 1$
\begin{enumerate}
	\item Select any $\varphi^r_m \in \D$ satisfying $F_{m-1}(\varphi^r_m) \ge t_m \|F_{m-1}\|_\D$;
	\item Find $\lambda_m \ge 0$ such that
		\[
			\|f^r_{m-1} - \la_m \varphi^r_m\|
			\le \Big(1 + \frac{\eta_m}{3}\Big) \inf_{\la \ge 0} \|f^r_{m-1} - \la \varphi^r_m\|;
		\]
	\item Find $\mu_m \in \mathbb{R}$ such that
		\[
			\|f - \mu_m (G^r_{m-1} + \la_m \varphi^r_m)\| 
			\le \Big(1 + \frac{\eta_m}{3}\Big) \inf_{\mu \in \mathbb{R}} \|f - \mu (G^r_{m-1} + \la_m \varphi^r_m)\|
		\]
		and define $G^r_m := \mu_m (G^r_{m-1} + \la_m \varphi^r_m)$;
	\item Set $f^r_m := f - G^r_m$.
\end{enumerate}
\smallskip
We note that steps $2$ and $3$ of the ARWRGA can be defined with two different error sequences and the corresponding results will hold with the appropriate minor modifications; however, for simplicity of presentation, we state those steps with the same error sequence $\{\eta_m\}_{m=1}^\infty$ with a factor of $\nicefrac{1}{3}$.

Next, we show that the stated algorithms are Approximate Biorthogonal Weak Greedy Algorithms.
\begin{Proposition}\label{prp:dga_awbga}
The AWCGA, the AWGAFR, and the ARWRGA belong to the class $\mathcal{AWBGA}$ with 
\[
	\e_m = \inf_{\lambda > 0} \frac{\delta_m + \eta_m + 2\rho(\lambda\|G_m\|)}{\lambda}.
\]
\end{Proposition}

\begin{proof}[Proof of Proposition~\ref{prp:dga_awbga}]
Throughout the proof we will only use the superscripts $c,e,r$ in the estimates that are specific for one of the algorithms, and omit superscripts when an estimate is applicable to any of the algorithms.

From the definitions of the corresponding algorithms it is clear that we only need to establish the sequence $\{\e_m\}_{m=1}^\infty$ for the biorthogonality condition~\ref{awbga_bo}.
Indeed, by the definition of the modulus of smoothness~\eqref{def:modulus_of_smoothness} we have for any $\lambda > 0$
\[
	\|f_m - \lambda G_m\| + \|f_m + \lambda G_m\| 
	\leq 2\|f_m\| \left( 1 + \rho \left( \frac{\lambda_m\|G_m\|}{\|f_m\|} \right)\right).
\]
Assume that $F_m(G_m) \geq 0$ (case $F_m(G_m) < 0$ is handled similarly). 
Then
\[
	\|f_m + \lambda G_m\| \geq F_m(f_m + \lambda G_m)
	\geq (1-\delta_m)\|f_m\| + \lambda F_m(G_m),
\]
and thus
\[
	\|f_m - \lambda G_m\| \leq 
	\|f_m\| \left( 1 + \delta_m + 2\rho \left( \frac{\lambda\|G_m\|}{\|f_m\|} \right)\right) - \lambda F_m(G_m).
\]
We now estimate $\|f_m - \lambda G_m\|$ for each of the algorithms.
From the corresponding definitions we obtain
\begin{align*}
	\|f^c_m - \lambda G^c_m\| 
	&\ge \inf_{G\in\Phi^c_m} \|f - G\| \ge (1+\eta_m)^{-1} \|f^c_m\| \ge (1-\eta_m) \|f^c_m\|,
	\\
	\|f^e_m - \lambda G^e_m\| 
	&\ge \inf_{\mu \ge 0} \|f - \mu G^e_m\| 
	\ge \inf_{\lambda \ge 0, w \in \mathbb{R}} \|f - ((1-w)G^e_{m-1} + \lambda\varphi^e_m)\|
	\\
	&\phantom{\ge \inf_{\mu \ge 0} \|f - \mu G^e_m\|\ }
	\ge (1+\eta_m)^{-1} \|f^e_m\| \ge (1-\eta_m) \|f^e_m\|,
	\\
	\|f^r_m - \lambda G^r_m\| 
	&\ge \inf_{\mu \ge 0} \|f - \mu G^r_m\| 
	= \inf_{\mu \ge 0} \|f - \mu (G^r_{m-1} + \lambda^r_m \varphi^r_m)\|
	\\ 
	&\phantom{\ge \inf_{\mu \ge 0} \|f - \mu G^r_m\|\ }
	\ge (1+\eta_m)^{-1} \|f^r_m\| \geq (1-\eta_m) \|f^r_m\|.
\end{align*}
Therefore
\[
	\lambda F_m(G_m) 
	\leq \|f_m\| \left( \delta_m + \eta_m + 2\rho \left( \frac{\lambda\|G_m\|}{\|f_m\|} \right)\right)
\]
and, since the inequality holds for any $\lambda > 0$,
\[
	F_m(G_m) \le \e_m 
	:= \inf_{\lambda>0} \frac{1}{\lambda} \left(\delta_m + \eta_m + 2\rho(\lambda\|G_m\|)\right).
\]
\end{proof}

\section{Convergence results for the $\mathcal{AWBGA}$}\label{sec:awbga_convergence}

In this section we investigate how the weakness and error parameters affect the convergence properties of the class $\mathcal{AWBGA}$.

First, we state the sufficient conditions on the weakness sequences that guarantee convergence of algorithms from the $\mathcal{AWBGA}$.
\begin{Theorem}\label{thm:awbga_convergence}
Let $X$ be a uniformly smooth Banach space with the modulus of smoothness $\rho(u)$ of power type $1 < q \le 2$.
Let sequences $\{t_m\}_{m=1}^\infty$, $\{\delta_m\}_{m=0}^\infty$, $\{\eta_m\}_{m=1}^\infty$, $\{\e_m\}_{m=1}^\infty$ be such that
\begin{gather}
	\label{theorem_AWBGA_t_k_conditions}
	\sum_{k=1}^\infty t_k^p = \infty,
	\\
	\label{theorem_AWBGA_error_conditions}
	\delta_{m-1} + \eta_m = o(t_m^p),
	\quad
	\e_m = o(1).
\end{gather}
Then an algorithm from the $\mathcal{AWBGA}$ with the weakness sequence $\tau$ converges for any target element $f \in X$.
\end{Theorem}

Next, we formulate the following analog of Theorem~\ref{thm:wbga_convergence_rate}, which states such bounds for the weakness and error sequences that the convergence rate of the $\mathcal{AWBGA}$ is the same as the one of the $\mathcal{WBGA}$.
\begin{Theorem}\label{thm:awbga_convergence_rate}
Let $X$ be a uniformly smooth Banach space with the modulus of smoothness $\rho(u)$ of power type $1 < q \le 2$. 
Let $f \in X$ be a target element.
Take a number $\e \ge 0$ and an $f^\e$ from $X$ such that
\[
	\|f - f^\e\| \le \e,
	\quad
	f^\e/A(\e) \in \A_1(\D)
\]
with some number $A(\e) > 0$.
Then an algorithm from the $\mathcal{AWBGA}$ with the weakness and error parameters $\{t_m\}_{m=1}^\infty$, $\{\delta_m\}_{m=0}^\infty$, $\{\eta_m\}_{m=1}^\infty$, $\{\e_m\}_{m=1}^\infty$, satisfying 
\begin{align}
	\label{thm:awbga_convergence_rate_condition1}
	&\delta_m + \e_m/\|f_m\| \leq 1/4,
	\\
	\label{thm:awbga_convergence_rate_condition2}
	&\delta_m + \eta_{m+1} \leq \frac{1}{2} C(q,\gamma)^{-p} A(\e)^{-p} t_{m+1}^p \|f_m\|^p
\end{align}
for any $m \ge 0$, provides
\[
	\|f_m\| \le \max\left\{4\e,\, C(q,\gamma) (A(\e)+\e) \Big(1 + \sum_{k=1}^m t_k^p\Big)^{-1/p}\right\},
\]
where $C(q,\gamma) = 4q (2\gamma)^q (\nicefrac{2}{q-1})^{1/p}$ and $p = q/(q-1)$.
\end{Theorem}

\begin{Corollary}\label{cor:awbga_convergence_rate}
Let $X$ be a uniformly smooth Banach space with the modulus of smoothness $\rho(u)$ of power type $1 < q \le 2$. 
Then for any $f \in \A_1(\D)$ an algorithm from the $\mathcal{AWBGA}$ with the weakness and error parameters $\{t_m\}_{m=1}^\infty$, $\{\delta_m\}_{m=0}^\infty$, $\{\eta_m\}_{m=1}^\infty$, $\{\e_m\}_{m=1}^\infty$, satisfying 
\begin{align*}
	&\delta_m + \e_m/\|f_m\| \leq 1/4,
	\\
	&\delta_m + \eta_{m+1} \leq \frac{1}{2} C(q,\gamma)^{-p} t_{m+1}^p \|f_m\|^p
\end{align*}
for any $m\ge 0$, provides
\[
	\|f_m\| \le C(q,\gamma)\Big(1+\sum_{k=1}^m t_k^p\Big)^{-1/p},
\]
where $C(q,\gamma) = 4q(2\gamma)^q (\nicefrac{2}{q-1})^{1/p}$ and $p = q/(q-1)$.
\end{Corollary}

Since Theorems~\ref{thm:awbga_convergence} and~\ref{thm:awbga_convergence_rate} involve generally unknown weakness parameters $\{\e_m\}_{m=1}^\infty$ and thus are not as straightforward as their analogues from Section~\ref{sec:wbga_convergence} (i.e. Theorems~\ref{cor:wbga_convergence} and~\ref{thm:wbga_convergence_rate}), we derive the following result for the algorithms from Section~\ref{sec:awbga}~--- the AWCGA, the AWGAFR, and the ARWRGA.

\begin{Remark}
Theorems~\ref{thm:awbga_convergence} and~\ref{thm:awbga_convergence_rate} are known for the AWCGA (see~\cite{T7} and~\cite{De}).
For the AWGAFR and the ARWRGA stated results are novel.
\end{Remark}

\begin{Proposition}\label{prp:awbga_convergence}
Let $X$ be a uniformly smooth Banach space with the modulus of smoothness $\rho(u)$ of power type $1 < q \le 2$. 
Let sequences $\{t_m\}_{m=1}^\infty$, $\{\delta_m\}_{m=0}^\infty$, $\{\eta_m\}_{m=1}^\infty$ be such that
\[
	\sum_{k=1}^\infty t_k^p = \infty,
	\qquad
	\delta_{m-1} + \eta_m = o(t_m^p).
\]
Then the AWCGA, the AWGAFR, and the ARWRGA converge for any dictionary $\D$ and any target element $f \in X$.
\\
Moreover, if $f\in \A_1(\D)$ and the sequences $\{\delta_m\}_{m=0}^\infty$ and $\{\eta_m\}_{m=1}^\infty$ satisfy
\begin{gather}
	\label{proposition_rate_conditions}
	\begin{aligned}
	\delta_m &\leq 64^{-p} \gamma^{1-p} \|f_m\|^p t_{m+1}^p,
	\\
	\eta_m &\leq 64^{-p} \gamma^{1-p} \|f_m\|^p t_m^p,
\end{aligned}
\end{gather}
then the AWCGA, the AWGAFR, and the ARWRGA provide
\[
	\|f_m\| \le C(q,\gamma)\Big(1+\sum_{k=1}^m t_k^p\Big)^{-1/p},
\]
where $C(q,\gamma) = 4q(2\gamma)^q (\nicefrac{2}{q-1})^{1/p}$ and $p = q/(q-1)$.
\end{Proposition}

\section{Proofs of results from Section~\ref{sec:awbga_convergence}}\label{sec:awbga_proofs}

First, we need the following analogue of Lemma~\ref{lem:wbga_er}.
\begin{Lemma}[{{\bf Error Reduction Lemma}}]\label{lemma_AWBGA}
Let $X$ be a uniformly smooth Banach space with the modulus of smoothness $\rho(u)$. 
Take a number $\e\ge 0$ and two elements $f$, $f^\e$ from $X$ such that
\[
	\|f - f^\e\| \le \e,
	\quad
	f^\e/A(\e) \in \A_1(\D)
\]
with some number $A(\e) \ge \e$.
Then for an algorithm from the $\mathcal{AWBGA}$ we have for any $m \geq 0$ and $\lambda \geq 0$
\begin{align*}
	\|f_{m+1}\| \leq \|f_m\| (1+\eta_{m+1}) \Bigg(
	1 &+ \delta_m + 2\rho\Big(\frac{\lambda}{\|f_m\|}\Big) 
	\\
	&- \lambda t_{m+1} A^{-1}(\e) \Big(1 - \delta_m - \frac{\e_m + \e}{\|f_m\|} \Big)\Bigg).
\end{align*}
\end{Lemma}

\begin{proof}[Proof of Lemma~\ref{lemma_AWBGA}]
The error reduction property~\ref{awbga_er} implies for any $\lambda \geq 0$
\[
	\|f_{m+1}\| \leq (1+\eta_{m+1}) \|f_m - \lambda \phi_{m+1}\|.
\]
From the definition of modulus of smoothness~\eqref{def:modulus_of_smoothness} we have
\[
	\|f_m - \lambda \phi_{m+1}\| \leq 2\|f_m\| \left(1 + \rho(\lambda/\|f_m\|)\right) - \|f_m + \lambda \phi_{m+1}\|
\]
and the greedy selection condition~\ref{awbga_gs} implies with Lemma~\ref{L2.2c}
\begin{align*}
	\|f_m + \lambda \phi_{m+1}\| 
	&\geq F_m(f_m + \lambda \phi_{m+1})
	= F_m(f_m) + \lambda F_m(\phi_{m+1})
	\\
	&\geq (1 - \delta_m) \|f_m\| + \lambda t_{m+1} \sup_{\phi \in \A_1(\D)} F_m(\phi)
	\\
	&\geq (1 - \delta_m) \|f_m\| + \lambda t_{m+1} A^{-1}(\e) F_m(f^\e).
\end{align*}
Assumption on $f$ and $f^\e$ and the biorthogonality property~\ref{awbga_bo} provide
\begin{align*}
	F_m(f^\e) 
	&\geq F_m(f) - \e = F_m(f_m + G_m) - \e
	\\
	&\geq (1 - \delta_m) \|f_m\| - \e_m - \e.
\end{align*}
Combining the above estimates we get
\begin{align*}
	\|f_{m+1}\| \leq (1+\eta_{m+1}) \Big(2\|f_m\| (1 &+ \rho(\lambda/\|f_m\|)) - (1 - \delta_m) \|f_m\|
	\\
	&- \lambda t_{m+1} A^{-1}(\e) \big((1 - \delta_m) \|f_m\| - \e_m - \e \big)\Big),
\end{align*}
which completes the proof.
\end{proof}

Now we can prove the convergence result for the $\mathcal{AWBGA}$.
\begin{proof}[Proof of Theorem~\ref{thm:awbga_convergence}]
Assume that for some $f \in \D$ an algorithm from the $\mathcal{AWBGA}$ does not converge, i.e. there exists $\alpha > 0$ and a sequence $\{m_\nu\}_{\nu=0}^\infty$ such that for any $\nu \geq 0$
\be\label{theorem_AWBGA_assumption}
	\|f_{m_\nu}\| > \alpha.
\ee
We will obtain a contradiction by showing that this assumption does not hold for any sufficiently big $\nu$.
\\
Fix $\e = \alpha/4$, take corresponding $f^\e$ and $A = A(\e)$, and denote 
\[
	\varkappa := \frac{\alpha^p}{2(16 A)^p (2\gamma)^{p-1}},
	\quad 
	p = \frac{q}{q-1}.
\]
Find such $K$ that for any $k \geq K$ the following estimates hold
\begin{align}
	\label{theorem_AWBGA_condition_1}
	&\delta_k + 2\e_k/\alpha \leq 1/4,
	\\
	\label{theorem_AWBGA_condition_2}
	&\delta_k + \eta_{k+1} \leq \varkappa t_{k+1}^p.
\end{align}
Note that existence of such $K$ is guaranteed by conditions~\eqref{theorem_AWBGA_error_conditions}.
\\
Take any $k \geq K$. 
If $\|f_k\| \leq \alpha/2$, then by the error reduction condition we have
\be\label{theorem_AWBGA_estimate_1}
	\|f_{k+1}\| \leq (1+\eta_{k+1}) \|f_k\| \leq 2 \|f_k\| \leq \alpha.
\ee
If $\|f_k\| > \alpha/2$, then Lemma~\ref{lemma_AWBGA} and condition~\eqref{theorem_AWBGA_condition_1} provide
\begin{align*}
	\|f_{k+1}\| 
	&\leq \|f_k\| (1+\eta_{k+1}) \Bigg(
	1 + \delta_k + 2\rho\Big(\frac{\lambda}{\|f_k\|}\Big) 
	- \frac{\lambda t_{k+1}}{A} \Big(
	1 - \delta_k - \frac{\e_k + \alpha/4}{\|f_k\|} \Big)\Bigg)
	\\
	&\leq \|f_k\| (1+\eta_{k+1}) \Bigg(
	1 + \delta_k + 2\gamma\Big(\frac{2\lambda}{\alpha}\Big)^q 
	- \lambda \frac{t_{k+1}}{4A} \Bigg).
\end{align*}
Let $\lambda_k$ be the positive root of the equation
\[
	\lambda \frac{t_{k+1}}{8A} = 2\gamma \left(\frac{2\lambda}{\alpha}\right)^q,
\]
which implies that
\[
	\lambda_k = \left( \frac{\alpha^q t_{k+1}}{16 \, 2^q A\gamma} \right)^{\frac{1}{q-1}}
	\ \text{ and }\ 
	\lambda_k \frac{t_{k+1}}{8A} 
	= \frac{\alpha^p t_{k+1}^p}{(16 A)^p (2\gamma)^{p-1}}
	= 2\varkappa t_{k+1}^p.
\]
The choice $\lambda = \lambda_k$ in the previous estimate guarantees
\[
	\|f_{k+1}\| \leq \|f_k\| (1+\eta_{k+1}) (1 + \delta_k - 2\varkappa t_{k+1}^p),
\]
and condition~\eqref{theorem_AWBGA_condition_2} provides
\be\label{theorem_AWBGA_estimate_2}
	\|f_{k+1}\| \leq \|f_k\| (1 - \varkappa t_{k+1}^p).
\ee
Note that estimates~\eqref{theorem_AWBGA_estimate_1} and~\eqref{theorem_AWBGA_estimate_2} guarantee that if $\|f_k\| \leq \alpha$ then $\|f_{k+1}\| \leq \alpha$.
By condition~\eqref{theorem_AWBGA_t_k_conditions} there exists such $N > K$ that 
\[
	\|f_K\| \prod_{k=K+1}^N (1 - \varkappa t_k^p) \leq \alpha.
\]
Therefore for any $\nu$ such that $m_\nu \geq N$ we have
\[
	\|f_{m_\nu}\| \leq \alpha,
\]
which contradicts the assumption~\eqref{theorem_AWBGA_assumption} and proves Theorem~\ref{thm:awbga_convergence}.
\end{proof}

We proceed with the proof of the estimate on the rate of convergence of the $\mathcal{AWBGA}$.
\begin{proof}[Proof of Theorem~\ref{thm:awbga_convergence_rate}]
Take any $k\ge 0$.
If $\|f_k\| \leq 2\e$, then by the error reduction condition we have
\be\label{thm:awbga_convergence_rate_fk<2e}
	\|f_{k+1}\| \leq (1+\eta_{k+1}) \|f_k\| \leq 2 \|f_k\| \leq 4\e.
\ee
If $\|f_k\| > 2\e$ then by Lemma~\ref{lemma_AWBGA} and condition~\eqref{thm:awbga_convergence_rate_condition1} we get for any $\lambda > 0$ 
\begin{align*}
	\|f_{k+1}\| 
	&\leq \|f_k\| (1+\eta_{k+1}) \Bigg(
	1 + \delta_k + 2\rho\Big(\frac{\lambda}{\|f_k\|}\Big) 
	- \frac{\lambda t_{k+1}}{A} \Big(
	1 - \delta_k - \frac{\e_k + \e}{\|f_k\|} \Big)\Bigg)
	\\
	&\leq \|f_k\| (1+\eta_{k+1}) 
	\Bigg(1 + \delta_k + 2\gamma\lambda^q\|f_k\|^{-q} - \frac{\lambda t_{k+1}}{4A}\Bigg).
\end{align*}
Hence, by taking infimum over all $\lambda > 0$ and using condition~\eqref{thm:awbga_convergence_rate_condition2}, we obtain
\begin{align}
	\nonumber
	\|f_{k+1}\| 
	&\leq \inf_{\lambda > 0} \|f_k\| (1+\eta_{k+1}) 
	\Bigg(1 + \delta_k + 2\gamma\lambda^q\|f_k\|^{-q} - \frac{\lambda t_{k+1}}{4A}\Bigg)
	\\
	\nonumber
	&= \|f_k\| (1+\eta_{k+1}) 
	(1 + \delta_k - (q-1) (4qA)^{-p} (2\gamma)^{1-p} t_{k+1}^p \|f_k\|^p)
	\\
	\nonumber
	&\leq \|f_k\| \Big(1 - \frac{1}{2} (q-1) (4qA)^{-p} (2\gamma)^{1-p} t_{k+1}^p \|f_k\|^p\Big)
	\\
	\label{thm:awbga_convergence_rate_fk>2e}
	&= \|f_k\| (1 - C^{-p} A^{-p} t_{k+1}^p \|f_k\|^p),
\end{align}
where
\[
	C = C(q,\gamma) := 4q (2\gamma)^q \Big(\frac{2}{q-1}\Big)^{1/p}.
\]
Therefore estimates~\eqref{thm:awbga_convergence_rate_fk<2e} and~\eqref{thm:awbga_convergence_rate_fk>2e} guarantee that as long as conditions~\eqref{thm:awbga_convergence_rate_condition1} and~\eqref{thm:awbga_convergence_rate_condition2} are satisfied, $\|f_k\|\le 4\e$ implies $\|f_{k+1}\|\le 4\e$.
Thus for any $m\ge 0$ either $\|f_m\|\le 4\e$ or $\|f_k\| > 4\e$ for any $k\le m$.
In the latter case, by using estimate~\eqref{thm:awbga_convergence_rate_fk>2e} and taking into account that $\|f\| \leq \|f^\e\| + \e \leq A(\e)+\e$ and $C > 1$ (due to the estimate $\gamma \ge 2^{-q}$ which follows from $u-1 \le \rho(u) \le \gamma u^q$), we complete the proof in the same way as in the proof of Theorem~\ref{thm:wbga_convergence_rate} in Section~\ref{sec:wbga_proofs}.
\end{proof}

Finally, we prove the convergence and the rate of convergence of the AWCGA, the AWGAFR, and the ARWRGA.
\begin{proof}[Proof of Proposition~\ref{prp:awbga_convergence}]
It is enough to verify that $\e_m$ satisfies conditions of Theorems~\ref{thm:awbga_convergence} and~\ref{thm:awbga_convergence_rate}.
Indeed, from Proposition~\ref{prp:dga_awbga} we get 
\begin{align*}
	\e_m 
	&\leq \inf_{\lambda>0} \frac{1}{\lambda} (\delta_m + \eta_m + 2\gamma\|G_m\|^q\lambda^q)
	\\
	&= q (q-1)^{-1/p} (\delta_m+\eta_m)^{1/p} (2\gamma)^{1/q} \|G_m\|
	\\
	&\leq 6(2\gamma)^{1/q} (\delta_m + \eta_m)^{1/p},
\end{align*}
since for all considered algorithms we have $\|G_m\| \leq \|f_m\| + \|f\| \leq 3\|f\| \leq 3$.
Therefore $\e_m \to 0 \text{ as } m \to \infty$ and Theorem~\ref{thm:awbga_convergence} guarantees convergence.
\\
Assume now that conditions~\eqref{proposition_rate_conditions} are satisfied. 
We will show that estimates~\eqref{thm:awbga_convergence_rate_condition1} and~\eqref{thm:awbga_convergence_rate_condition2} hold as well.
Indeed, using the estimates $\gamma \ge 2^{-q}$ and $\|f_m\| \le 2\|f\| \le 2$ we obtain
\[
	\delta_m + \e_m / \|f_m\| 
	\leq 64^{-p} + 6(2\gamma)^{1/q} \big(2 \,64^{-p} \gamma^{1-p}\big)^{1/p}
	= 64^{-p} + \frac{3}{16} < \frac{1}{4}.
\]
Similarly, since $\|f_{m+1}\| \le 2\|f_m\|$, we get
\begin{align*}
	\delta_m + \eta_{m+1} 
	&\leq (2^{-p}+1) 32^{-p} \gamma^{1-p} \|f_m\|^p t_{m+1}^p
	\\
	&\leq 16^{-p} (2\gamma)^{1-p} t_{m+1}^p \|f_m\|^p
	\\
	&\leq \frac{q-1}{2} (4q)^{-p} (2\gamma)^{1-p} t_{m+1}^p \|f_m\|^p.
\end{align*}
Hence Theorem~\ref{thm:awbga_convergence_rate} guarantees the stated rate of convergence.
\end{proof}

\section{$X$-greedy algorithms}\label{sec:x-greedy}

Algorithms from the class $\mathcal{WBGA}$ perform the greedy selection step based on the norming functional $F_{f_m}$.
For this reason, this type of algorithms is called {\it dual} greedy algorithms.
As mentioned in Section~\ref{sec:overview}, there exists another natural class of greedy algorithms~--- the $X$-greedy type algorithms~--- that solve a simple optimization problem rather than compute the norming functional.
Thus, in certain applications, such algorithms might be more suitable for implementation than the dual greedy algorithms.

Generally one can obtain an $X$-greedy algorithm from a given dual algorithm by replacing the greedy selection step
\[
	\varphi_m = \mathop{\argmin}_{g \in \D} F_{f_{m-1}}(g)
\]
with a minimization over all possible choices of a newly added element from the dictionary
\[
	(\varphi_m, \lambda_m) = \mathop{\argmin}_{g \in \D, \lambda \in \mathbb{R}} \|f_{m-1} - \lambda g\|.
\]

It is also known that convergence results for such $X$-greedy companions can be derived from the proofs of the corresponding results for the dual greedy algorithms.
While the exact connection between dual and $X$-greedy algorithms and the detailed analysis of the latter is outside of the scope of this paper, we illustrate the concept by defining and deriving convergence results for the Rescaled Relaxed $X$-Greedy Algorithm (RRXGA)~--- the $X$-greedy companion for the RWRGA.
For a more detailed discussion on $X$-greedy algorithms and related theoretical results see~\cite{VT83} and~\cite{Tbook}.

\bigskip\noindent
{\bf Rescaled Relaxed $X$-Greedy Algorithm (RRXGA)\\}
Set $f_0 := f$ and $G_0 := 0$.
Then for each $m \ge 1$
\begin{enumerate}
	\item Select any $\varphi_m \in \D$ and $\lambda_m \in \mathbb{R}$ (we assume existence) satisfying
		\[
			\|f_{m-1} - \lambda_m \varphi_m\|
			= \inf_{g \in \D, \lambda \in \mathbb{R}} \|f_{m-1} - \lambda g\|;
		\]
	\item Find $\mu_m \in \mathbb{R}$ such that
		\[
			\|f - \mu_m (G_{m-1} + \lambda_m \varphi_m)\| 
			= \inf_{\mu \in \mathbb{R}} \|f - \mu (G_{m-1} + \lambda_m \varphi_m)\|
		\]
		and define $G_m := \mu_m (G_{m-1} + \la_m \varphi_m)$;
	\item Set $f_m := f - G_m$.
\end{enumerate}
\smallskip

Next, we state one of the corresponding convergence results.
\begin{Theorem}\label{T4.1n}
Let $X$ be a uniformly smooth Banach space with the modulus of smoothness $\rho(u)$ of power type $1 < q \le 2$.
Let $f \in X$ be a target element.
Take a number $\e \ge 0$ and an $f^\e$ from $X$ such that
\[
	\|f - f^\e\| \le \e,
	\quad
	f^\e/A(\e) \in \A_1(\D)
\]
with some number $A(\e) > 0$.
Then the RRXGA provides
\[
	\|f_m\| \le \max\left\{2\e, C(q,\gamma) (A(\e)+\e) (1+m)^{-1/p}\right\},
	\quad p = q/(q-1),
\]
where $C(q,\gamma) = 4(2\gamma)^{1/q}$.
\end{Theorem}

In order to prove the stated theorem, we need the following lemma, which is a generalization of the Error Reduction Lemma~\ref{lem:wbga_er} and can be largely derived from its proof.
\begin{Lemma}[{{\bf Generalized Error Reduction Lemma}}]\label{lem:wbga_ger}
Let $X$ be a uniformly smooth Banach space with the modulus of smoothness $\rho(u)$.
Take a number $\e\ge 0$ and two elements $f$, $f^\e$ from $X$ such that
\[
	\|f - f^\e\| \le \e,
	\quad
	f^\e/A(\e) \in \A_1(\D)
\]
with some number $A(\e) \ge \e$.
Suppose that $f$ is represented as $f = f' + G'$ in such a way that $F_{f'}(G') = 0$ and an element $\ff' \in \D$ is chosen to satisfy 
\[
	F_{f'}(\ff') \ge \theta \|F_{f'}\|_\D,
	\quad
	\theta \in [0,1].
\] 
Then we have
\[
	\inf_{\la \ge 0} \|f' - \la \ff'\| \le \|f'\|\inf_{\la\ge0}
	\left(1 - \la \theta A(\e)^{-1} \left(1 - \frac{\e}{\|f'\|}\right) + 2\rho\left(\frac{\la}{\|f'\|}\right)\right).
\]
\end{Lemma}

\begin{proof}[Proof of Theorem~\ref{T4.1n}]
It is sufficient to prove an analog of the Error Reduction Lemma~\ref{lem:wbga_er} for the RRXGA.
Let $G_m$ and $f_m$ be the approximant and the residual of the RRXGA at the $m$-th iteration.
Using Lemma~\ref{lem:wbga_ger} with $G' = G_m$ and $f' = f_m$ find such $\ff'\in\D$ that the generalized error reduction holds.
Hence it follows from the definition of the RRXGA that for the residual $f_{m+1}$ we have
\begin{align*}
	\|f_{m+1}\|
	&\le \inf_{\la \ge 0} \|f_m - \la \ff_{m+1}\|
	\le \inf_{\la \ge 0} \|f' - \la \ff'\|
	\\
	&\le \|f_m\| \inf_{\la \ge 0} 
	\left(1 - \la \theta A(\e)^{-1} \left(1 - \frac{\e}{\|f_m\|}\right) 
	+ 2\rho\left(\frac{\la}{\|f_m\|}\right)\right).
\end{align*}
Taking supremum over all $\theta \le 1$ gives the Error Reduction Lemma for the RRXGA.
We complete the proof of Theorem~\ref{T4.1n} in the same way as we derived Theorem~\ref{thm:wbga_convergence_rate} from Lemma~\ref{lem:wbga_er}.
\end{proof}

\section{Discussion}\label{sec:discussion}

This paper is devoted to a theoretical study of fundamental techniques used in sparse representation of data, which can be structured or unstructured.
The contemporary challenge is unstructured data, which comes from different sources such as medical, engineering, and networks, among others.
In order to apply a sparsity based method one needs to use a dictionary (a representation system), which provides sparse representation for the data at hand.
In some cases the data structure itself gives us an idea of which dictionary to use.
For instance, in signal processing, say, in music processing, it is natural to use the trigonometric system or the Gabor system as a dictionary for sparse representation.
In other cases, especially in the case of unstructured data, we need to learn a dictionary providing sparse representation for the given data.
Dictionary learning is an important and rapidly developing area of numerical mathematics.
Clearly, we cannot expect that a dictionary learnt from given unstructured data will have some special structure.
Therefore, the theory broadly applicable to Big Data problems must address the problem of sparse representation with respect to an arbitrary (structured and unstructured) dictionary.
This motivates our setting with greedy approximation with respect to an arbitrary dictionary. 

In order to address the contemporary needs of data managing, a very general model of approximation with regard to a redundant system (dictionary) has been considered in many recent papers.
As such a model, we choose a Banach space $X$ with elements as target functions and an arbitrary system $\D$ of elements of this space such that the closure of $\sp\D$ coincides with $X$ as a representation system.
We would like to have an algorithm of constructing $m$-term approximants that adds at each step only one new element from $\D$ and keeps elements of $\D$ obtained at the previous steps.
This requirement is an analogue of {\it on-line} computation that is desirable in practical algorithms.
Clearly, we are looking for good algorithms which converge for each target function.
It is not obvious that such an algorithm exists in a setting at the above level of generality ($X$, $\D$ are arbitrary).

An important argument that motivates us to study this problem in the infinite dimensional setting is that in many contemporary data management applications the dimensionality $d$ of an ambient space $\R^d$ is quite large, which negatively affects the rate of convergence of many native algorithms (the so-called {\it curse of dimensionality}).
As greedy algorithms are naturally designed for use in infinitely-dimensional spaces, our results provide bounds on the convergence rate that are independent of the dimensionality of the space.

While in many applications users are satisfied with a Hilbert space setting, we comment on our interest in a more general Banach space setting.
An {\it a-priori} argument for a Banach space approach is that the spaces $L_p$ are very natural in many real-life applications and hence they should be studied along with the $L_2$ space.
An {\it a-posteriori} argument is that the study of greedy approximation in Banach spaces has discovered that the behavior of greedy algorithms is governed by the smoothness of the space that is generally described by the modulus of smoothness $\rho(u)$.
It is known that the spaces $L_p$ for $2\le p<\infty$ have modulo of smoothness of the power type $2$, i.e. the same as a Hilbert space.
Thus, many results that are known for the Hilbert space $L_2$ and proved using some special structure of a Hilbert space can be generalized to Banach spaces $L_p$, $2\le p<\infty$ as they only rely on the geometry of the unit sphere of the space expressed in the form $\rho(u) \le \gamma u^2$.
Also, we note that in the case $L_p$-space the implementation of a dual greedy algorithm is not substantially more difficult than in $L_2$ since the norming functionals have the following simple explicit form
\[
	F_f(g) = \int\limits_{\Omega} \frac{\sign(f) |f|^{p-1} \, g}{\|f\|_p^{p-1}} \,\mathrm{d}\mu.
\]
Thus at the greedy selection step of the algorithm one has to find an approximate solution to the optimization problem
\[
	\text{maximize}\ \ 
	\int\limits_{\Omega} \sign(f_m) |f_m|^{p-1} \, g \,\mathrm{d}\mu
	\ \ \text{s.t.}\ \ 
	g \in \D,
\]
which is not much more difficult in the case $p \neq 2$ than in the case $p = 2$.

Finally, we comment on the issue of numerical stability.
Since the selection step of a greedy algorithm is usually the most expensive (as we need to search over the whole dictionary), there is a natural desire to simplify it by relaxing the search criteria, which results in a {\it weak} greedy algorithm, where instead of finding $\sup_{g\in\D} F(g)$ we are satisfied with an element $\ff\in\D$ such that 
\[
	F(\ff) \ge t_m \sup_{g\in\D} F(g).
\]
Theorem~\ref{thm:wbga_convergence_rate} demonstrates a major (and surprising) advantage of the Weak Biorthogonal Greedy Algorithms~--- the weak version of an algorithm from the $\mathcal{WBGA}$ with the weakness sequence $\tau = \{t\}$, $t\in (0,1]$ has the same convergence properties as the strong version of the algorithm with $t=1$.
Furthermore, Theorem~\ref{thm:awbga_convergence} guarantees in this case the stable convergence under very mild conditions on the relative errors~--- they must go to zero with $m \to \infty$.

We have carried over a thorough study of three fundamental properties of the Weak Biorthogonal Greedy Algorithms~--- convergence, rate of convergence, and computational stability.
The above discussion suggests that algorithms from the $\mathcal{WBGA}$ are easily implementable in many Banach spaces of interest (e.g. $L_p$) and possess the necessary theoretical properties that are essential for realization in real-life applications.

\section*{Acknowledgements}
The work was supported by the Russian Federation Government Grant No. 14.W03.31.0031.

The paper contains results obtained in frames of the program ``Center for the storage and analysis of big data'', supported by the Ministry of Science and High Education of Russian Federation (contract 11.12.2018N{\textsuperscript{\underline{o}}}13/1251/2018 between the Lomonosov Moscow State University and the Fond of support of the National technological initiative projects).

We are grateful to the reviewer for useful comments which resulted in an improved structure of the paper.

\bibliographystyle{elsarticle-harv}
\bibliography{references}

\end{document}